\documentclass[a4paper,twoside,12pt]{article}
\usepackage[english]{babel}

\usepackage{amsmath,amsfonts,amssymb,amsthm}
\newtheorem{teo}{Theorem}
\newtheorem{lem}{Lemma}

\newtheorem{rem}{Remark}

 \title{{\bf  Natural  methods of the  unsupervised  topological alignment  }}

\author{ Maksim \,V.~Kukushkin,  Mikhail S. Arbatskiy,  Dmitriy E. Balandin,\\ Alexey V. Churov   \\ \\
  \small  \textit{Pirogov Russian National Research Medical University,}\\
    \small  \textit{Russian Clinical and Research Center of Gerontology,}\\
    \small  \textit{ Ministry of Health of the Russian Federation,  129226, Moscow, Russia}\\
  \small\textit{Russian Academy of Sciences,  Kabardino-Balkarian Scientific Center,}\\
 \small\textit{Institute of Applied Mathematics and Automation, 360000,  Nalchik, Russia}\\}

\date{}
\begin{document}
\maketitle

\begin{abstract} In the  paper,    we represent a comparison analysis  of the  methods of the topological alignment and extract the main mathematical principles forming the base of the concept. The main narrative is devoted to the so-called  coupled methods dealing with  the data sets of  various nature. As a main theoretical result, we obtain  harmonious generalizations of the graph Laplacian  and kernel based methods with the central idea to find a natural structure coupling  data sets of various nature. Finally,  we discuss prospective applications and  consider far reaching generalizations  related to the hypercomplex numbers.

\end{abstract}
\begin{small}\textbf{Keywords:}
  Manifold alignment; Unsupervised topological alignment
of single-cell multi-omics integration; Graph Laplacian;     Single cell RNA-seq; Single cell epigenomics.   
\end{small}

\section{Introduction}

  In many papers the method of coupling images of heterogeneous data sets is based on invention  of a special technique involving artificial constructions at least not reflecting at most contradicting  to the mathematical nature of the mapping, i.e. saying  "mathematical nature" we mean abstract  classical mathematical notions such as the space, the algebraic structure, etc. It is clear that dealing with the classical mathematical notions we have an opportunity to develop the theory harmoniously using well-known methods and putting them as a fundamental base inventing new ones that can be in their own turn classified in accordance with the accepted terminology. Quite the contrary is the case  when we involve artificial technical formulas allowing us to achieve a concrete goal but likely preventing the development of the theory  in general, for in this case we know nothing on the consequent influence of the used construction on the scheme of reasonings and it is rather reasonable, due to the technical obstacles,  to expect  the end of the theory.

Manifold alignment is a class of algorithms that create a mapping of heterogeneous data sets into a common, lower-dimensional latent space. The central point is to find a mapping that reveals the entire structure of the initial manifold.  It is clear that a metric reflecting the distance plays a great role in the  issue  as well as a mathematical nature of the mapping, however the latter should  be harmoniously connected with applications. Diagonal integration represents a joint analysis of multi-omics biological data - genomic, transcriptomic, proteomic, and metabolomic. A critical view of existing methods and tools for diagonal integration must be supported by experience in analyzing individual data types.
The scientific team  has analyzed genomic data to identify damaging mutations \cite{Mironenko2022}, scRNA-seq transcriptomic data of cell lines \cite{Arbatskiy2024}, proteomic data \cite{Kulebyakina2024}, small non-coding RNA data \cite{Basalova2020}, as well as metabolomic profiles. The accumulated experience of working with various types of omics data creates a foundation for developing methods for their integration.

Experimental methods on individual cells reveal transcriptomic and epigenetic heterogeneity between cells, but the question on their relationship remains open.     Note  that various types of topological alignment include the principal  concept of the relationship between images or sometimes between preimages. The relationship is provided by  an artificial structure generally represented  by an operator with the idea to spread a unified  structure on  the data sets of the different nature. The literary  survey represented below produces arguments creating prerequisites  for the comprehensive investigation.

 In the paper \cite{Joshua2017},  an approach for integrating several types of measurements of observed data on a single cell is considered. However, the stated results are not based on a detailed theory, moreover the concept of connecting heterogeneous data is not represented. In the paper \cite{Jie Liu 2019}, the authors represent  the  manifold alignment method that is used to integrate several types of measurements of observed data performed on various aliquots of a given cell population. The MMD-MA (Maximum Mean Discrepancy Manifold Alignment) method minimizes the discrepancy between images  of different modalities in a common latent space using the maximum mean discrepancy as a measure of the dissimilarity that   provides a theoretically sound approach.
However, minimizing the discrepancy  can lead to loss of the information specific to each modality  since  a smooth representation   does not reflect  the real heterogeneity of the biological data. Moreover, the connection between heterogeneous data sets is based on the mathematical construction obtained by virtue of the optimization, where the artificially constructed penalty functions are involved, what   does not  fully preserve the inherent structure of the data sets.\\
In the paper \cite{KaiCao2022}, the authors represent   a modification of the Gromov-Wasserstein distance-based manifold  alignment method \cite{Chapel2020}, which combines heterogeneous multiomic data sets of a single cell in order to describe and represent common and data set-specific cellular structures in different modalities.
The Pamona (Partial Gromov-Wasserstein distance-based manifold alignment) method applies partial optimal alignment of manifolds allowing not all elements to be compared between modalities, but only those for which reliable correspondences can be found. However, the method  requires an additional criterion  to determine the threshold for the reliability of correspondences and may lead to the loss of potentially informative connections due to  the specific  approach to establishing correspondences.\\
In the paper \cite{Duren2018}, the authors combine two clustering processes so that in clustering cells corresponding to scRNA-seq sample the information from the scATAC-seq sample could also be used and vice versa. The authors  formulate this
problem of coupled clustering as an optimization problem and present
a method for solving it called   coupled non-negative
matrix factorization.\\
In the paper  \cite{KaiCao2020},   a new algorithm for unsupervised topological alignment
of multiomic integration for a single cell is represented. The method does not require any information on   correspondence between cells or between measurements. At the first stage, the method embeds the internal low-dimensional structure of each data set of a single cell into a matrix, the elements of which reflect the distances corresponding to cells within the same data set, and then aligns the cell images  by
comparison  of distance matrices using the matrix optimization method. Finally, the method projects separate  non-comparable measurements
corresponding to  single cell data sets into a latent  embedding space to provide comparability of measurements by virtue of  aligned cell images.\\
In the paper \cite{Dou2020}, the topological alignment method is formulated for two modalities, where  a modality is understood as a set of features for a sample of cells,
the information on the features is reflected in a matrix, the columns of which represent
cells, the rows represent various features. Since the objects in  two data sets  are  different, the  additional technical tool is required for the method the so-called transition matrix connecting the heterogeneous data sets.   The transition  matrix   is understood as a projection of one data set onto the  space corresponding to another data set. The   transition matrix is obtained from the profiles of scATAC-seq due to   the information given in the bodies of genes
 \cite{Wang 2020,Stuart2019,Korsunsky2019,Rosenberg2018}.\\
 Summarizing the  descriptions given above, we come to the conclusion that in each case there is a  technically involved   unnatural  structure providing  a relationship between the images of the data sets. Generally, a rather abstract term  "relationship" means the property of preservation of the proportion between the distances in the initial and latent spaces. However, from the mathematical point of view, the constructed mapping should inherit the abstract properties of the involved structure. Thus,  classical mathematical structures, such as the structures  defined on the Hilbert space or generated by the properties of well-known operators acting in the Hilbert space (here we use  a notion including  Euclidean spaces), generate  mappings having   classical properties. The corresponding methods of the topological alignment are called   natural. On the other hand,  the practical relevance of the methods considered above  appears just due to mathematical tools  applied in some complicated order that may not be harmonious from the fundamental theory point of view. The corresponding methods of the topological alignment are called   unnatural. In this paper, we study  natural methods such as Laplacian graph and kernel based methods of the topological alignment. Generally, we consider a natural algebraic  structure such as the finite-dimensional unital algebra of the hypercomplex numbers and produce   concrete reasonings corresponding to the field of the complex numbers showing the relevance of the approach. The real and imaginary parts of the complex number correspond to the coupled data sets respectively, thus the construction  of a mapping defined on the complex Hilbert space includes a naturally coupled structure.   The main achievement  is an attempt to construct an abstract qualitative theory creating an opportunity to consider a number of heterogeneous data sets endowed with the coupled natural structure,  let alone far reaching modifications and generalizations.\\
The paper is organized as follows: in section 2, we consider notations and some well-known facts; in section 3, we introduce the idea of  the coupled unsupervised manifold alignment; in paragraph 3.1, we consider a coupled Laplacian  mapping and prove  auxiliary propositions;  in paragraph 3.2, we consider the unsupervised manifold alignment via the reproducing kernel; in paragraph 3.3, we consider prospective applications to the modern methods; in section 3, we finalize the theoretical results and discuss prospective generalizations.

\section{Preliminaries}

Throughout the paper, we consider   finite dimensional matrices  over the field of the complex numbers $\mathbb{C}^{n\times m},$  i.e. of the dimension $(n\times m),\;n,m\in \mathbb{N},$  and use the following notation $
A =\{a_{sj}\}\in \mathbb{C}^{n\times m}, \,s=1,2,\ldots,n,\;j=1,2,\ldots,m
$  for a matrix. We use the following standard notation for the   transpose operation
$
A^{T}=\{a_{js}\}.
$
 Consider a matrix
$
A= \{a_{sj}\}\in \mathbb{C} ^{n\times  m},\;a_{sj}\in \mathbb{C},
$
denote
$$
\mathbf{a}^{\cdot}_{s }:=
 ( a_{s1}, a_{s2},\ldots,a_{sm})^{T},\; \mathbf{a}_{ j } = (a_{1j}, a_{2j},\ldots,a_{nj})^{ T }.
$$
We consider a complex linear vector space  $\mathbb{C}^{n}$ consists of the set of column matrices  whose elements are complex numbers $\mathbf{x}=(x_{1},x_{2},\ldots,x_{n})^{T},\; x_{j} \in \mathbb{C},\,j=1,2\ldots,n.$
Analogously, we define the real  linear vector space  $\mathbb{R}^{n}.$
A complex linear vector space supplied with the given below structure of the   inner (scalar ) product operation is called   the complex Euclidean space
$$
(\mathbf{x},\mathbf{y})_{\mathbb{C}^{n}}:=\sum\limits_{j=1}^{n}x_{j}\bar{y}_{j},\;\mathbf{x},\mathbf{y}\in \mathbb{C}^{ n}.
$$
  Note that the complex Euclidean space represents a particular case of the unitary space, i.e.  a linear space over the field of the complex numbers.    Define the norm in the general sense  in a space with the inner  product as follows
$
\|\mathbf{x}\|=\sqrt{(\mathbf{x},\mathbf{x})}.
$
Consider a   matrix $A =\{a_{sj}\}\in \mathbb{C}^{n\times m},$ denote $\bar{A} =\{\bar{a}_{sj}\}.$
In accordance with   the orthogonal decomposition theorem, we have
$$
\mathbb{C}^{n}=\mathbb{C}^{p} \stackrel{\cdot}{+}\mathbb{C}^{n-p},\,n>p,
$$
i.e. for an arbitrary element     $\mathbf{x}\in \mathbb{C}^{n}$ there exists a unique pair  $\mathbf{x}_{1} \in \mathbb{C}^{p},\,\mathbf{x}_{2}\in \mathbb{C}^{n-p}$ of elements so that the following decomposition holds $
 \mathbf{x} =\mathbf{x}_{1} +\mathbf{x}_{2}.
$
In this way we can put the element  $\mathbf{x}_{1}$ into correspondence with the element   $\mathbf{x},$ which is called the orthogonal projection  of the  space     $\mathbb{C}^{n}$ onto the subspace  $\mathbb{C}^{p}.$
 Given an arbitrary set
$X,$  a totally ordered set
$Y,$  and a function $f:X\rightarrow Y,$  the
$\mathrm{argmin}$
  over some subset
$\Omega \subseteq  X$  is defined by the following expression

$$
  \underset{  x \in \Omega    }{ \mathrm{argmin}} f(x):=\{y\in \Omega:\,f(y)\leq f(x),\,x\in \Omega\}.
 $$
Consider a set of the elements
 $
 \{\mathbf{x}_{1},\mathbf{x}_{2},\ldots,\mathbf{x}_{n} \}
 $
 belonging to  a normed space.
Let us construct a weighted graph $G$ having  $n$ vertices, where one vertex corresponds to one element, and a set of edges connecting the adjacent vertices of the graph.
We enumerate the vertices according to the given order, i.e. $j$ vertex corresponds to $\mathbf{x}_{j}.$  We suppose that the vertices    $s$ and $j$ are adjacent to each other if the elements   $\mathbf{x}_{s}$ and  $\mathbf{x}_{j}$ are related,   for instance close in some sense. The question  is how to define the rather vague notion of  closeness in a more concrete way \cite{Belkin 2004}.

Here, we consider two variants, the first one is the so-called $\varepsilon$ - neighborhood method which postulates the following sense of closeness.  Suppose that the vertices    $s$ and $j$ are adjacent to each other  if
 $$
 \|\mathbf{x}_{s}-\mathbf{x}_{j} \|   <\varepsilon,
 $$
where the norm is understood in the abstract sense. However, the Euclidean norm  is mostly used  in applications.
Note that this type of closeness is geometrically motivated.
The relationship is naturally symmetric. However, it often
leads to a graph  with several connected components and it is  difficult
to choose the convenient value of $\varepsilon$ to avoid the unconnected graph.

The second variant   is the so-called $n$ - nearest neighbors method which postulates the following sense of closeness.  Suppose that the vertices    $s$ and $j$ are adjacent to each other  if  $\mathbf{x}_{s}$ belongs to the set of  $n$ nearest,  in the sense of the norm neighbors,  of the element   $\mathbf{x}_{j}$ or  $\mathbf{x}_{j}$ belongs to the set of   $n$ nearest  neighbors of the element
 $\mathbf{x}_{s}.$

  Having constructed the graph in accordance with   one of the above methods, we can choose weights to form the weight matrix. The first one relates to the heat kernel, we assume that if the vertices $s$ and  $j$ are connected, in symbol $s\sim j,$ then the weight matrix is formed from the elements
$$
W_{sj}=e^{- t^{-1}\|\mathbf{x}_{s}-\mathbf{x}_{j}\|^{2} },\,t\in \mathbb{R} \setminus \{0\},
$$
in the contrary case, we assume that  $W_{sj}  = 0.$ Thus, we have
$$
W_{sj}=\left\{ \begin{aligned}
  e^{- t^{-1}\|\mathbf{x}_{s}-\mathbf{x}_{j}\|^{2}},\;s\sim j \\
 0,\;s\nsim j   \\
\end{aligned}
 \right.  .
$$
In accordance with the  definition, we obtain the adjacency matrix of the weighted graph
$
W= \{W_{sj}\}\in \mathbb{C}^{n\times n}.
$
The justification for choosing this weight can be found  in the papers \cite{Ham2005}, \cite{Ek2009}. The second one  is the so-called simple weight, we assume that $W_{sj}  = 1$ if the  vertices $s$ and  $j$ are connected and $W_{sj}  = 0$ if   $s$ and
$j$ are disconnected, i.e.
$$
W_{sj}=\left\{ \begin{aligned}
  1,\;s\sim j \\
 0,\;s\nsim j   \\
\end{aligned}
 \right.  .
$$
This kind of simplicity gives the opportunity to avoid choosing the value of the parameter    $t.$  Note that a matrix $A\in \mathbb{C}^{n\times m}$ generates the finite-dimensional operator
$
A:  \mathbb{C}^{m}\rightarrow \mathbb{C}^{n}.
$
We use the following notation for the Hermitian components of the operator
$$
 \mathfrak{Re}A=\frac{A+A^{\ast}}{2},\;\mathfrak{Im}A=\frac{A-A^{\ast}}{2i},\;\bar{A}:=\{\bar{a}_{sj} \},\;\mathrm{Re} (A):=\{\mathrm{Re}a_{sj} \},\;\mathrm{Im} (A):=\{\mathrm{Im}a_{sj} \},
$$
 the latter matrices are called  the real and imaginary part of the matrix $A$ respectively.
The detailed information on the Hermitian components properties is given in the paper \cite{firstab_lit Math2024}. Denote by    $    \mathrm{D}   (A),\,   \mathrm{R}   (A),\,\mathrm{N}(A)$      the    domain of definition, the  range,  and the  kernel or   null space  of the  operator $A$ respectively.

\section{Coupled unsupervised manifold alignment }

The manifold alignment represents a solution to the problem
of alignment and at the same time forms the basis for finding a unified   representation of different  data sets.
The principal  concept of manifold alignment is  to use the relationships between objects
within each data set to obtain the information on the relationships between data sets and
eventually to create a mapping of initially different  data sets into a common  latent space.  The approaches described in this paragraph deal with different data sets  having  the same underlying structure. The basic low-dimensional
representation is extracted by modeling the local geometry in the complex vector space using the generalized graph Laplacian  operator associated to the  data sets naturally coupled on the complex plane.
Apparently, the alignment of manifolds can be considered as a method of reducing the dimension of the latent space \cite{Ham2003},
where the goal is to find a low-dimensional embedding of different data sets that preserves
any known correspondences between them. Although we consider unsupervised methods, they can be easily modified to   semi-supervised or supervised methods due to the  corresponding term reflecting the distance from the chosen initial point in the sense of the given norm. In the reduced simplified form, we can formulate  the coupled  problem of the dimension  decreasing  as follows.  Given  $n$ elements
$$
\mathbf{x}_{1},\mathbf{x}_{2},\ldots,\mathbf{x}_{n}\in \mathbb{C}^{l},\;\mathbf{x}_{j}=(x_{1j}, x_{2j},\ldots,x_{lj})^{T},\; x_{sj} \in \mathbb{C},\;l\in \mathbb{N},
$$
we are challenged to find a set of the elements
$$
\mathbf{y}_{1},\mathbf{y}_{2},\ldots,\mathbf{y}_{n}\in \mathbb{C}^{m},\,1\leq m\leq l,
$$
so that the element   $\mathbf{y}_{j}$ is the image of the element $\mathbf{x}_{j}$ under the searched mapping, i.e.
$
\mathbf{x}_{j}\rightarrow \mathbf{y}_{j}.
$
 Moreover, we are interested in finding a mapping for  which the images would be located close to each other in some sense. In the case corresponding to the Euclidean space
the desired mapping can be created due to the subspace generated by the generalized eigenvectors of the graph Laplacian operator. Below, we represent the detailed analysis of the algorithm solving this problem in the case corresponding to the complex Euclidean space which  simultaneously solves the problem of coupling  different data sets.

\subsection{Coupled Laplacian  mapping}

Throughout the paper, we consider a   symmetric   matrix $W\in \mathbb{C}^{ n\times n },\,W=\{W_{sj}\}$ having nonnegative real and imaginary parts,
define a diagonal matrix
$$
D=\{D_{sj}\},\;D_{sj}=\left\{ \begin{aligned}
  \sum\limits_{k=1}^{n}W_{kj},\;s=j \\
 0,\;s\neq j   \\
\end{aligned}
 \right.  .
$$
Further we assume that $D_{ss}\neq 0,\;s=1,2,\ldots,n$ which does not restrict the appropriate  class of matrices due to the given  sense of the elements.
 Consider an operator  $L:=D-W.$ Note that $L$ is not selfadjoint, the verification is left to the reader. However,  it has the following property $L^{\ast}=\bar{L}.$
Since  $ D$   is a diagonal matrix, then  it commutes with an  arbitrary matrix of the corresponding dimension. In accordance with the symmetry property, we have
$$
\mathrm{Re}L_{jj}=\sum\limits_{k\neq j}\mathrm{Re} W_{kj},\;j=1,2,\ldots,n,
$$
therefore the matrix $\mathrm{Re}(L)$ is   symmetric  diagonally dominant. Hence, it is  positive semi-definite, i.e. $\mathrm{Re}(L)\geq 0,$ whence it follows that   $\mathfrak{Re}L\geq 0.$ The analogous reasoning  leads  to the fact    $\mathfrak{Im}L\geq 0.$
The following lemma establishes a technical tool relevant in the  further reasonings.
\begin{lem}\label{L1} Assume that $W\in \mathbb{C}^{ n\times n },\;X\in \mathbb{C}^{n\times m} $   then  the following relation holds
$$
  \mathrm{tr}\{ X ^{ \ast } L  X \} =\frac{1}{2} \sum\limits_{s,j=1}^{n}\|\mathbf{x}^{\cdot}_{s}-\mathbf{x}^{\cdot}_{j}\|_{\mathbb{C}^{m}}^{2}W_{sj}.
$$
\end{lem}
\begin{proof}
It is clear that
\begin{equation}\label{1y}
\mathrm{tr}\{ X ^{ \ast } L  X \}=\sum\limits_{q=1}^{m} (X ^{ \ast } L  X)_{qq}=\sum\limits_{q=1}^{m} ( L \mathbf{x} _{ q },    \mathbf{x} _{ q })_{\mathbb{C}^{n}}.
\end{equation}
On the other hand
$$
\mathrm{Re} \sum\limits_{s,j=1}^{n}|x_{sq}-x_{jq}|^{2}W_{sj}=
\mathrm{Re}\left\{\sum\limits_{s,j=1}^{n}\left(|x_{sq}|^{2}+|x_{jq}|^{2}-2\mathrm{Re}\{x_{sq}\bar{x}_{jq}\}\right)W_{sj}\right\}=
$$
\begin{equation*}\label{1b}
= 2\mathrm{Re}\left\{\sum\limits_{s=1}^{n}\mathrm{Re}D_{ss} |x_{sq}|^{2} - \sum\limits_{s,j=1}^{n} x_{sq}\bar{x}_{jq} \mathrm{Re}W_{sj}\right\}
=2 (\mathbf{x}_{q} ,  \mathfrak{Re} L \mathbf{x} _{q})_{\mathbb{C}^{n}}=2 \mathrm{Re} (L \mathbf{x} _{q},  \mathbf{x} _{q})_{\mathbb{C}^{n}}.
 \end{equation*}
Analogously, we get
$$
\mathrm{Im} \sum\limits_{s,j=1}^{n}|x_{sq}-x_{jq}|^{2}W_{sj}=\mathrm{Im}\left\{ \sum\limits_{s,j=1}^{n}\left(|x_{sq}|^{2}+|x_{jq}|^{2}-2\mathrm{Re}\{x_{sq}\bar{x}_{jq}\}\right)W_{sj}\right\}=
$$
\begin{equation*}\label{2b}
 =\mathrm{Re}\left\{\sum\limits_{s=1}^{n}\mathrm{Im}D_{ss} |x_{sq}|^{2} - \sum\limits_{s,j=1}^{n} x_{sq}\bar{x}_{jq} \mathrm{Im}W_{sj}\right\} = 2\mathrm{Re}(\mathbf{x} _{q},   \mathfrak{Im } L\mathbf{x} _{q})_{\mathbb{C}^{n}}=2\mathrm{Im}(L\mathbf{x} _{q},      \mathbf{x} _{q})_{\mathbb{C}^{n}}.
\end{equation*}
Combining the above  relations, we get
 $$
\sum\limits_{q=1}^{m}       ( L\mathbf{x} _{ q },       \mathbf{x} _{ q })_{\mathbb{C}^{n}}=\frac{1}{2}\sum\limits_{q=1}^{m} \sum\limits_{s,j=1}^{n}W_{sj}|x_{sq}-x_{jq} |^{2}= \frac{1}{2} \sum\limits_{s,j=1}^{n}W_{sj}\sum\limits_{q=1}^{m}|x_{sq}-x_{jq} |^{2}=\frac{1}{2} \sum\limits_{s,j=1}^{n}\|\mathbf{x}^{\cdot}_{s}-\mathbf{x}^{\cdot}_{j}\|_{\mathbb{C}^{m}}^{2}W_{sj}.
$$
Using formula \eqref{1y}, we obtain   the desired result.
\end{proof}

\begin{teo}\label{T1}
Assume that   $X \in \mathbb{C}^{n\times n}$, the following condition is satisfied
\begin{equation}\label{cond_A}
X^{\ast}DX = e^{i\theta}\cdot I.
\end{equation}
Then the following relation holds
$$
 \mathrm{tr}\{ X ^{ \ast } L  X \} =  e^{i\theta} \cdot \mathrm{tr}\{D^{-1}L\}.
$$
\end{teo}

\begin{proof}
Let $\mathbf{x} _{j}$ denote the $j$-th column of the matrix $X$. First, recall the well-known coordinate representation of the matrix trace. For an arbitrary matrix $M \in \mathbb{C}^{n \times n}$, the sum of quadratic forms over the columns of $X$ satisfies
\begin{equation}\label{trace_prop}
\sum\limits_{j=1}^{n} (M\mathbf{x}_{j}, \mathbf{x} _{j})_{\mathbb{C}^{n}} = \sum\limits_{j=1}^{n} (M\mathbf{x} _{j}, \mathbf{x} _{j})_{\mathbb{C}^{n}} = \mathrm{tr}\{ X^{\ast} M X \}.
\end{equation}
Applying relation \eqref{trace_prop} for $M = L$, we can write
\begin{equation}\label{step1}
 \mathrm{tr}\{ X ^{ \ast } L  X\} = \sum\limits_{j=1}^{n} (L\mathbf{x} _{j}, \mathbf{x} _{j})_{\mathbb{C}^{n}} .
\end{equation}

Now let us analyze the structural properties of the matrix $X$. Since the matrix $D$ is invertible and $\det(e^{i\theta}I) \neq 0$, equality  \eqref{cond_A} implies that $X$ is invertible. To prove this fact it is sufficient to consider the determinants of both sides of the equality, we have $\left|
\det X\right|^{2}|\det D|=\left|\det X^{\ast}DX\right|=\left|\det e^{i\theta}I\right|  =1.$
Multiplying both sides of $X^{\ast}DX = e^{i\theta}I$ by $X^{-1}$ from the right, we obtain
$$
X^{\ast}D = e^{i\theta} X^{-1}.
$$
Taking the inverse of both sides of this equality, we get
$$
(X^{\ast}D)^{-1} = (e^{i\theta} X^{-1})^{-1} \implies D^{-1} (X^{\ast})^{-1} = e^{-i\theta} X.
$$
  Multiplying by $X^{\ast}$ from the right, we obtain
\begin{equation}\label{AA_prop}
X X^{\ast} = e^{i\theta} D^{-1}.
\end{equation}
Using the cyclic property of the trace operation, we rewrite the right-hand side of \eqref{step1} as follows
$$
\mathrm{tr}\{ X^{\ast} L X \} = \mathrm{tr}\{ L X X^{\ast} \}.
$$
Substituting expression \eqref{AA_prop} directly into the latter relation, we get
$$
\mathrm{tr}\{ L X X^{\ast} \} = \mathrm{tr}\{ L \cdot (e^{i\theta} D^{-1}) \} = e^{i\theta} \cdot \mathrm{tr}\{ L D^{-1} \}.
$$
Finally, applying the cyclic property of the trace,    we obtain
$$
\mathrm{tr}\{ X^{\ast} L X \} = e^{i\theta} \cdot \mathrm{tr}\{ D^{-1} L \}.
$$
 The proof is complete.
\end{proof}

\subsubsection{Generalized eigenvectors}

Having generalized the corresponding selfadjoint operator \cite{Belkin2003}, we will call the operator $L$  the generalized graph Laplacian operator. Consider the following generalized eigenvalue problem
\begin{equation}\label{4y}
 L\mathbf{e}=\lambda  B\mathbf{e},\quad \mathbf{e}\in \mathbb{C}^{n},\quad \lambda\in \mathbb{C},
\end{equation}
where $L$ is the generalized graph Laplacian operator, $B$ is an arbitrary invertible matrix. Consider the following formula deduced by direct calculation
\begin{equation}\label{5i}
     ( L\mathbf{x} ,       \mathbf{x}  )_{\mathbb{C}^{n}} =\frac{1}{2} \sum\limits_{s,j=1}^{n} | x _{s}- x _{j} | ^{2}W_{sj}.
\end{equation}
This gives us a  technically significant  fact
$
\mathrm{arg}( L\mathbf{x} ,       \mathbf{x}  )_{\mathbb{C}^{n}}\in (0,\pi/2),\;\mathbf{x}\in \mathbb{C}^{n}.
$
Note that relation \eqref{5i} shows that   the kernel of the coupled  graph Laplacian is a one dimensional subspace of $\mathbb{C}^{n}$ defined as follows
$$
\mathrm{N}(L)=\{\mathbf{x}\in \mathbb{C}^{n}:\;x_{j}=a,\;a\in \mathbb{C},\,j=1,2,\dots,n\}.
$$
Indeed the inclusion $(a,a,\ldots,a)^{T}\in \mathrm{N}(L)$ can be verified by direct calculations. The inverse inclusion follows from formula \eqref{5i}.
 Taking into account the fact $\mathrm{N}(L)=\mathrm{N}(L^{\ast}),$ we obtain
\begin{equation}\label{8x}
\mathbb{C}^{n}=\mathrm{N}(L) \oplus \mathrm{R}(L).
\end{equation}
If the matrix $B$ is positive definite so that we can involve the inner product generated by the operator $B,$ we have a decomposition
\begin{equation*}\label{8x}
\mathbb{C}^{n}=\mathrm{N}(L) \oplus_{B} \mathrm{N}(L)^{\perp(B)} .
\end{equation*}
The fact is that $\mathrm{N}(L)^{\perp(B)}\neq\mathrm{N}(L)^{\perp}=\mathrm{R}(L).$
  Apparently,  to avoid the trivial solution of the  optimization problem involving the operator $L,$    we should  involve these subspaces in the problem.

Further, unless otherwise stated, we assume  that   the generalized eigenvalue problem \eqref{4y} possesses a  $B$ - orthonormal   system of generalized eigenvectors $\mathbf{e}_1, \mathbf{e}_2, \dots, \mathbf{e}_{k+1},\,\;1\leq k\leq n-1$ corresponding to the non-zero eigenvalues $\lambda_1, \lambda_2, \dots, \lambda_{k},$ if the following relation holds
\begin{equation}\label{6y}
(B\mathbf{e}_{s},\mathbf{e}_{j})_{\mathbb{C}^{n}} =   \delta_{sj}, \quad s,j=1,2,\dots,k.
\end{equation}
The zero eigenvalue $\lambda_{k+1}=0$ corresponds to the  element  belonging to the kernel, i.e. $\mathbf{e}_{k+1}\in \mathrm{N}(L).$
 Consider the set of the generalized eigenvectors numbered in accordance with the order of the corresponding eigenvalues
$$
|\lambda_{1}| \leq |\lambda_{2}| \leq \dots \leq |\lambda_{k}|,\;\lambda_{k+1}   = 0.
$$
To establish a rigorous mathematical foundation for the optimization problem, we guarantee  the completeness and the structural inflexibility of the concept  by considering a more general approach to the optimization problem. We formalize the spectral requirements on the operators $L$ and $B$ in the form of the following basic structural assumption.

\begin{teo}\label{T2}
  Let
\(B\in\mathbb{C}^{n\times n}\) be invertible. Then the generalized eigenvalue
problem
\[
L\mathbf{e}=\lambda B\mathbf{e}
\]
possesses a complete \(B\)-orthonormal system of eigenvectors  if and only if
\begin{enumerate}
\item[(i)] the  operator $B $  is Hermitian and  positive definite;
\item[(ii)] the operator $B^{-1}L$    is normal with respect to the inner product
generated by the operator $B.$
\end{enumerate}
\end{teo}

\begin{proof}
\medskip
\noindent {\it Necessity.}
Assume that there exists a complete \(B\)-orthonormal system
$\{\mathbf{e}_s\}_{s=1}^{n }$ of eigenvectors of
\(L\mathbf{e}=\lambda B\mathbf{e}\) corresponding to the eigenvalues
$\lambda_s,$ including the  zero eigenvalue so that
$$
L\mathbf{e}_s=\lambda_s B\mathbf{e}_s, \qquad
(B\mathbf{e}_s,\mathbf{e}_t)=\delta_{st},\quad s,t=1,\dots,n.
$$
 Let us show that $B$ is Hermitian positive definite.
For an arbitrary $\mathbf{x}\in \mathbb{C}^{n},$ we have a  representation
$$
\mathbf{x}=\sum_{s=1}^{n} c_s \mathbf{e}_s.
$$
Therefore
$$
(B\mathbf{x},\mathbf{x})_{\mathbb{C}^{n}}
= \sum_{s,t=1}^{n} c_s \overline{c_t}\, (B\mathbf{e}_s,\mathbf{e}_t)_{\mathbb{C}^{n}}
= \sum_{s=1}^{n} |c_s|^2 .
$$
Hence the value $(B\mathbf{x},\mathbf{x})_{\mathbb{C}^{n}}$ is real and strictly positive for
\(\mathbf{x}\neq 0\). Using the   polarization identity,  it is not hard to prove that
$$(B\mathbf{x},\mathbf{y})_{\mathbb{C}^{n}}=\overline{(B\mathbf{y},\mathbf{x})}_{\mathbb{C}^{n}} .$$
Therefore the operator $B $  is Hermitian and  positive definite.

 Let us show that the operator  $A:=B^{-1}L$ is normal in the  Hilbert space generated by the set $\mathbb{C}^{n}$ and the inner product $(B\cdot,\cdot)_{\mathbb{C}^{n}}.$
It is clear that
$A\mathbf{e}_s=\lambda_s\mathbf{e}_s,\;s=1,2,...,n.$ Using this fact, we can prove that  in the basis
$\{\mathbf{e}_s\}_{1}^{n}$ the matrix of \(A\) is diagonal with entries
$\lambda_s,$ i.e.   $A=\mathrm{diag}\{\lambda_{1},\lambda_{2},\ldots,\lambda_{n-1},0\}.$ Since $\{\mathbf{e}_s\}_{1}^{n}$ is orthonormal with respect to the inner product $(B\cdot,\cdot)_{\mathbb{C}^{n}},$  the adjoint \(A^\#\) of \(A\) with respect to
this inner product has the diagonal matrix with entries
$ \bar{\lambda}_s ,$ i.e.  $A^\#=\mathrm{diag}\{\bar{\lambda}_{1},\bar{\lambda}_{2},\ldots,\bar{\lambda}_{n-1},0\}.$ Therefore, in the same basis, we have
$$A A^\# =\operatorname{diag}\bigl(|\lambda_1|^2,|\lambda_2|^2\dots,|\lambda_{n-1}|^2,0\bigr)
= A^\# A.
$$
Thus, in accordance with the definition $A$ is normal.

\medskip
\noindent {\it Sufficiency.} Assume that conditions  (i) and (ii) hold. In accordance with (i), the form  $(B\mathbf{x},\mathbf{y})_{\mathbb{C}^{n}}$ defines  a genuine inner
product in  $\mathbb{C}^{n}.$  In accordance with  (ii), the operator \(A:=B^{-1}L\)  is normal with respect to this inner product. Applying  the spectral
theorem for normal operators on a finite-dimensional complex inner product
space, we conclude that there exists an orthonormal basis $\{\mathbf{e}_s\}_{s=1}^{n} \in \mathbb{C}^{n}$ consisting of eigenvectors of the operator  $A,$ i.e.
$$
B^{-1}L\mathbf{e}_s=\lambda_s\mathbf{e}_s, \;
 (B\mathbf{e}_s,\mathbf{e}_t)_{\mathbb{C}^{n}}=\delta_{st},\;s,t=1,2,\ldots,n.
$$
The latter can be rewritten in the form
$$
 L\mathbf{e}_s=\lambda_sB\mathbf{e}_s, \;
 (B\mathbf{e}_s,\mathbf{e}_t)_{\mathbb{C}^{n}}=\delta_{st},\;s,t=1,2,\ldots,n.
$$
Therefore \(\{\mathbf{e}_s\}_{s=1}^{n}\) is a complete \(B\)-orthonormal
system of eigenvectors of the generalized problem. The proof is complete.
\end{proof}

Consider a manifold defined as follows $X^{\ast}BX= I_{m}.$ In the case when $B$ is positive definite, i.e. in symbol $B\succ0,$ we have a set homeomorphic to the Stiefel manifold and therefore that is smooth and  compact. Thus, in this case the optimization problem is well defined and we have all necessary good properties of the object to study the problem. However, the theorem represented below admits modifications useful in  studying    optimization problems regardless on the mentioned topological properties.

 Denote by $\{\beta_{j}\}_{1}^{\xi},\;\{\gamma_{j}\}_{1}^{\xi}$ the set of the eigenvalues $\{\lambda_{j}\}_{1}^{\xi}$ enumerated as follows
$$
\,\forall  \{ \beta_{j},\gamma_{j}\}, \exists \{\lambda_{p},\lambda_{q}\}:\;\mathrm{Re} \beta_{j}=\mathrm{Re} \lambda_{p},\,\mathrm{Im}\gamma_{j}=\mathrm{Im} \lambda_{q},\;j,p,q\in\{1,2,\ldots,\xi\} ,
$$
$$
 \mathrm{Re} \beta_{j}\leq \mathrm{Re} \beta_{j+1} ,\;\mathrm{Im} \gamma_{j}\leq \mathrm{Im} \gamma_{j+1},\,j=1,2,\ldots,\xi-1.
$$

\begin{teo}\label{T3}Assume that the generalized eigenvalue problem possesses a complete B-orthonormal  system of the eigenvectors, $X^{\ast}BX= I_{m}.$ Then   the following relation holds
\begin{equation}\label{7y}
\min \sum\limits_{j=1}^{m}\mathrm{Re}(L \mathbf{x} _{j},\mathbf{x} _{j})_{\mathbb{C}^{n}}=
\sum\limits_{j=1}^{m} \mathrm{Re}  \beta_{j}   ,\;
\min \sum\limits_{j=1}^{m}\mathrm{Im}(L \mathbf{x} _{j},\mathbf{x} _{j})_{\mathbb{C}^{n}}=\sum\limits_{j=1}^{m}\mathrm{Im} \gamma_{j}, \,X^{\ast}B\mathbf{1}=0 .
\end{equation}
\begin{equation}\label{12p}
\max \sum\limits_{j=1}^{m}\mathrm{Re}(L \mathbf{x} _{j},\mathbf{x} _{j})_{\mathbb{C}^{n}}=
\sum\limits_{j=n-m}^{n-1} \!\mathrm{Re}  \beta_{j}   ,\;
\max \sum\limits_{j=1}^{m}\mathrm{Im}(L \mathbf{x} _{j},\mathbf{x} _{j})_{\mathbb{C}^{n}}=\sum\limits_{j=n-m}^{n-1}\!\mathrm{Im} \gamma_{ j}.
\end{equation}
\end{teo}
\begin{proof}  Consider the minimum problem. In accordance with Theorem \ref{T2}, the operator $B$ is Hermitian and positive definite. Hence,   we have the $B$-orthogonal decomposition
$$
\mathbb{C}^{n}=\mathrm{N}(L)\oplus_{B}\mathrm{N}(L)^{\bot} , \;
\mathrm{N}(L)=\mathrm{span}\{\mathbf{1}\}.
$$
The condition   $X^{\ast}B\mathbf{1}=0 $ gives us $\mathbf{x}_{j}\in \mathrm{N}(L)^{\bot},\;j=1,2,\ldots,m.$
 Therefore, we have a decomposition
$$
\mathbf{x} _{j}= \sum\limits_{s=1}^{\xi} c_{sj }\mathbf{e}_{s},\;\xi:=n-1.
$$
Thus,  the following representation holds
$$
 \sum\limits_{j=1}^{m}( L\mathbf{x} _{j}  ,     \mathbf{x} _{j}   )_{\mathbb{C}^{n}}=  \sum\limits_{j=1}^{m}\left( \sum\limits_{s=1}^{\xi}c_{sj } \beta_{s}
         B\mathbf{e}_{s},\sum\limits_{s=1}^{\xi} c_{sj } \mathbf{e}_{s} \right)_{\! \mathbb{C}^{n}}=
          \sum\limits_{s=1}^{\xi}\beta_{s} \sum\limits_{j=1}^{m}  |c_{sj }|^{2},
$$
where according to the   agreement, we have $\beta_{1}<\beta_{2}<...<\beta_{\xi}.$
 It is clear that the real and imaginary components of the   eigenvalues   are positive numbers. It follows directly from the location of the numerical range of $L.$ By direct calculation, we get
\begin{equation*}
 \delta_{jk}=(  B \mathbf{x} _{j} ,\mathbf{x} _{k}   )_{\mathbb{C}^{n}}=
 \left(\sum\limits_{s=1}^{\xi}c_{sj }B\mathbf{e}_{s},\sum\limits_{s=1}^{\xi} c_{sk }\mathbf{e}_{s} \right)_{\!\!\!\mathbb{C}^{n}}=
 \sum\limits_{s=1}^{\xi} c _{sj }\overline{c} _{sk }.
\end{equation*}
Applying this formula, rewriting the left-hand side   of the first formula  \eqref{7y}, we come to the related  problem
$$
\min\limits \sum\limits_{s=1}^{\xi}\mathrm{Re}\beta _{s} \sum\limits_{j=1}^{m}  |c_{sj }|^{2},
\;\;\sum\limits_{j=1}^{m}\sum\limits_{s=1}^{\xi} |c _{sj }|^{2} =m.
$$
Denote
$$
u_{s}:=\sum\limits_{j=1}^{m}  |c_{sj }|^{2},\;s=1,2,\ldots,\xi.
$$
Then  the expression corresponding to the  main problem can be rewritten in the form
$$
d_{\xi}(\mathbf{u}) :=\sum\limits_{s=1}^{\xi}\mathrm{Re}\beta_{s}u_{s}  ,\;\;\sum\limits_{s=1}^{\xi}u_{s} =m,\;\mathbf{u}=( u_{1}, u_{2},\ldots,u_{\xi}).
$$
Substituting, we get
$$
h_{\xi-1}(\mathbf{u}):=\sum\limits_{s=1}^{\xi-1}u_{s}\mathrm{Re}\beta_{s}+\mathrm{Re}\beta_{\xi}\left(m-\sum\limits_{s=1}^{\xi-1}u_{s} \right)=
\sum\limits_{s=1}^{\xi-1}u_{s} \left(\mathrm{Re}\beta_{s}-\mathrm{Re}\beta_{\xi}\right) +m \mathrm{Re}\beta_{\xi},
$$
$$
\sum\limits_{s=1}^{\xi-1}u_{s}\leq m.
$$
Therefore, the components of the gradient $\nabla h_{\xi-1}$  are negative, i.e.  $\mathrm{Re}\beta_{s}-\mathrm{Re}\beta_{\xi}<0.$ It can be proved that under the assumption that $B$ is positive definite,
the set defined by the connection condition $X^{\ast}BX=I_m$ is homeomorphic
to the Stiefel manifold and therefore is compact. Since  the functional $h_{\xi-1}$ is linear, its minimum is attained on the
boundary. The negative components of the gradient  imply that  $u_{\xi}$ attains
its lower bound, i.e. $u_{\xi}=0$, equivalently
$$
\sum\limits_{s=1}^{\xi-1}u_{s}=m.
$$
Therefore
$$
\min \limits_{\mathbf{u}\in \mathbb{R}^{\xi-1}_{m}}  h_{\xi-1}(\mathbf{u})= \min\limits_{\mathbf{u}\in \mathbb{T}^{\xi-1}_{m}}\sum\limits_{s=1}^{\xi-1}u_{s}\mathrm{Re}\beta_{s}=  \min\limits_{\mathbf{u}\in \mathbb{T}^{\xi-1}_{m}}  d_{\xi-1}(\mathbf{u}),
$$
where
$$
\mathbb{R}^{n}_{m}:=\left\{\mathbf{x}\in \mathbb{R}^{n}:\, \sum\limits_{s=1}^{n}x_{s}\leq m \right\},\;\mathbb{T}^{n}_{m}:=\left\{\mathbf{x}\in \mathbb{R}^{n}:\, \sum\limits_{s=1}^{n}x_{s}= m \right\},\;n,m\in \mathbb{N}.
$$
Apparently, in finding the minimum point, we have some freedom in choosing the coefficients, thus at this step of the proof, we   assume    $c_{\xi j}=0,\,j=1,2,\ldots,m.$
Following  the same reasoning, we get
\begin{equation*}
h_{\xi-2}(\mathbf{u}):=
\sum\limits_{s=1}^{\xi-2}u_{s} \left(\mathrm{Re}\beta _{s}-\mathrm{Re}\beta _{\xi-1}\right)+m \mathrm{Re}\beta _{\xi-1},\;\sum\limits_{s=1}^{\xi-2}u_{s}\leq m.
\end{equation*}
Therefore, the constant vector $\nabla h_{\xi-2}$ has negative coordinates. Analogously to the above, we get
$$
\min\limits_{\mathbf{u}\in \mathbb{R}^{\xi-2}_{m}}  h_{\xi-2}(\mathbf{u})=   \min\limits_{\mathbf{u}\in \mathbb{T}^{\xi-2}_{m}}  d_{\xi-2}(\mathbf{u}),
$$
where the latter relation is obtained due to the choice  $c_{\xi-1j}=c_{\xi j}=0,\,j=1,2,\ldots,m.$
Following the same reasoning, we come to the problem
$$
\min\limits_{X^{\ast}BX= I_{m}} \sum\limits_{j=1}^{m}\mathrm{Re}(L\mathbf{x}_{j},\mathbf{x}_{j})_{\mathbb{C}^{n}}=\min\limits_{ C ^{\ast} C =I_{m}} \sum\limits_{s=1}^{m}\mathrm{Re}\beta_{s} \sum\limits_{j=1}^{m}  |c_{sj }|^{2},
$$
where we assume  $c_{m+1j}=c_{m+2j}=\ldots=c_{\xi j}=0,\,j=1,2,\ldots,m.$
Therefore, we can restrict the problem considering the   relation
\begin{equation}\label{11}
      \sum\limits_{s=1}^{m}\mathrm{Re}\beta_{s} \sum\limits_{j=1}^{m}  |c_{sj }|^{2}, \;
\sum\limits_{s=1}^{m} c _{sj }\overline{c }_{sk }  =\delta_{kj},\,j,k=1,2,\ldots,m.
\end{equation}
 In  matrix form, we have
$
 C ^{\ast} C =I_{m},\;  \Rightarrow  C C ^{\ast}=I_{m},\;C:=\{c_{sj}\}\in \mathbb{C}^{m\times m},
$
i.e.,
$$
\sum\limits_{j=1}^{m} c _{sj  }\overline{c }_{kj } =\delta_{sk},\,s,k=1,2,\ldots,m.
$$
Substituting   into \eqref{11}, we obtain \eqref{7y}. The proof of formula \eqref{12p} is completely analogous. We only remark that
in the maximum problem the constraint $X^{\ast}B\mathbf{1}=0$ is no longer
necessary.
 The expression corresponding to the  main problem can be  written in the form
$$
d_{\xi}(\mathbf{u}) :=\sum\limits_{s=1}^{\xi}\mu _{s}u_{s}  ,\;\;\sum\limits_{s=0}^{\xi}u_{s}=m,\;\mathbf{u}=(u_{0},u_{1},\ldots,u_{\xi}),
$$
where
$$
u_{0}:= \sum\limits_{j=1}^{m}|c _{nj }|^{2},\;
\mathbf{x} _{j}= \sum\limits_{s=1}^{n} c_{sj }\mathbf{e}_{s},\;\mu_{s}=\beta_{ n-s} ,\,s=1,2,\ldots,\xi.
$$
Substituting, we come to the expression
$$
h_{\xi-1}(\mathbf{u}):=
\sum\limits_{s=1}^{\xi-1}u_{s} \left(\mu _{s}-\mu_{\xi}\right)-u_{0}\mu _{\xi}+m \mu _{\xi},\;
\sum\limits_{s=0}^{\xi-1}u_{s}\leq m.
$$
Therefore, the component of the gradient $\nabla h_{\xi-1}$ corresponding to   $u_{0}$ is   negative, whereas the  other  coordinates are positive  $ \mu_s - \mu_{\xi} > 0.$ Since the connection condition $X^{\ast}BX=I_m$ defines a compact set and  the functional $h_{\xi-1}$ is linear, its maximum is attained on the
boundary. The negative component of the gradient with respect to $u_{0}$,
$-\mu_{\xi}<0$, implies that  $u_{0}$ attains
its lower bound, i.e. $u_{0}=0$, which
eliminates the kernel contribution.
  The following part of the proof is based on the consequent   exclusion of the variables, analogously to the minimum problem considered above.
\end{proof}

The following lemma establishes the fundamental connection between the component-wise minimization problems resolved in Theorem \ref{T3} and the global minimization of the complex trace modulus.

\begin{lem}\label{L2}
Assume that the generalized eigenvalue problem possesses a complete B-orthonormal  system of the eigenvectors, moreover
$$
\mathrm{Re} \beta_{j}=\mathrm{Re} \lambda_{p},\,\mathrm{Im}\gamma_{j}=\mathrm{Im} \lambda_{q},\;j,p,q\in\{1,2,\dots,m\}.
$$
Then the following relation holds
\begin{equation}\label{argmin_modulus_final}
 \min  \left| \sum\limits_{j=1}^{m} (L\mathbf{x} _{j},\mathbf{x} _{j})_{\mathbb{C}^{n}} \right| = \left| \sum\limits_{j=1}^{m} \lambda_{j }   \right|,\,X^{\ast}BX= I_{m},\,X^{\ast}B\mathbf{1}=0.
\end{equation}
\end{lem}
\begin{proof}
By expanding the square of the modulus of the complex trace into its real and imaginary parts, we evaluate the lower bound as follows
$$
\left| \sum\limits_{j=1}^{m} (L\mathbf{x}_{j},\mathbf{x}_{j})_{\mathbb{C}^{n}} \right|^2 = \left( \mathrm{Re}\sum\limits_{j=1}^{m} (L\mathbf{x}_{j},\mathbf{x}_{j})_{\mathbb{C}^{n}} \right)^2 + \left( \mathrm{Im}\sum\limits_{j=1}^{m} (L\mathbf{x}_{j},\mathbf{x}_{j})_{\mathbb{C}^{n}} \right)^2.
$$
Applying the algebraic properties of the minimum operator, excluding the elements belonging to the kernel,  we obtain
\begin{equation}\label{ineq_step_sq}
\min\limits_{X^{\ast}B X = I_{m}} \left| \sum\limits_{j=1}^{m} (L\mathbf{x}_{j},\mathbf{x}_{j})_{\mathbb{C}^{n}} \right|^2 \geq
\left( \min\limits_{X^{\ast}B X = I_{m}} \mathrm{Re}\sum\limits_{j=1}^{m} (L\mathbf{x}_{j},\mathbf{x}_{j})_{\mathbb{C}^{n}} \right)^2 + \left( \min\limits_{X^{\ast}B X = I_{m}} \mathrm{Im}\sum\limits_{j=1}^{m} (L\mathbf{x}_{j},\mathbf{x}_{j})_{\mathbb{C}^{n}} \right)^2.
\end{equation}
Let us apply the exact optimization results of Theorem \ref{T3}, we get
\begin{equation}\label{perm_step}
\left( \min\limits_{X^{\ast}B X = I_{m}} \mathrm{Re}\sum\limits_{j=1}^{m} (L\mathbf{x}_{j},\mathbf{x}_{j})_{\mathbb{C}^{n}} \right)^2 + \left( \min\limits_{X^{\ast}B X = I_{m}} \mathrm{Im}\sum\limits_{j=1}^{m} (L\mathbf{x}_{j},\mathbf{x}_{j})_{\mathbb{C}^{n}} \right)^2 =
$$
$$
 =\left( \sum\limits_{j=1}^{m} \mathrm{Re}  \beta_{j} \right)^2 + \left( \sum\limits_{j=1}^{m} \mathrm{Im}  \gamma_{j} \right)^2=\left|\sum\limits_{j=1}^{m}    \lambda_{j} \right|^2.
\end{equation}
However, we have
$$
\sum\limits_{j=1}^{m} (L\mathbf{e}_{j},\mathbf{e}_{j})_{\mathbb{C}^{n}}=\sum\limits_{j=1}^{m}    \lambda_{j},\;Y^{\ast}BY=I_{m},\,Y^{\ast}B\mathbf{1}=0,
$$
where $Y  = (\mathbf{e}_{1}, \mathbf{e}_{2}, \ldots, \mathbf{e}_{m}).$    This establishes the exact equality in \eqref{ineq_step_sq} and completes the proof.
\end{proof}

\begin{rem}
In many applications it is convenient to work not on the range
$\mathrm{R}(L)$ but on a natural invariant subspace of the operator
$A:=B^{-1}L$. Namely, consider the root subspace corresponding to the
non-zero spectrum of the operator
$$
\mathfrak{R}(A)=\bigoplus_{\lambda\neq0}\ker(A-\lambda I)^{m_{\lambda}}.
$$
This subspace is invariant under $A$ by definition, and its dimension
equals the sum of algebraic multiplicities of the non-zero eigenvalues. In the
diagonal case, in particular in the symmetric case considered
below, it coincides with $\mathrm{R}(L).$  Thus, we can reformulate the criterion in the following form.

A complete $B$-orthonormal system of eigenvectors in the root vector subspace
exists if and only if $B$ is Hermitian positive definite and $A$ is normal with respect to the
$B$-inner product.  The proof is analogous to Theorem \ref{T2}, with
one essential simplification: the invariance of the root vector subspace  is a well-known fact and therefore  no additional assumptions on $B$ are needed.

The advantage over using the operator range as an initial subspace can be described as follows. First, the invariance
of the root vector subspace is given, whereas
the invariance of the range under $B$ requires an extra
assumption on $B$ that in general fails. Second, in the diagonal
case the considered  subspaces coincide, so the subsequent results remain valid
without modification.
\end{rem}

\subsubsection{Restriction of the dimension} \label{3.1.2}

Consider the sets of elements
\begin{equation}\label{9ys}
\mathbf{u}_{1},\mathbf{u}_{2},...,\mathbf{u}_{n}\in \mathbb{R}^{l},\;\;\mathbf{v}_{1},\mathbf{v}_{2},...,\mathbf{v}_{n}\in \mathbb{R}^{p},\;\;l,p\in \mathbb{N},
\end{equation}
and assume that the graph $G$ is constructed in accordance with one of the methods described in the introduction section, i.e. $\varepsilon$ - neighborhood or $n$ - nearest neighbors methods. Thus, we can put    weighted matrices  $W^{(1)},\,W^{(2)}$ into correspondence with the sets in accordance with the given order, where
\begin{equation*}\label{14xs}
\;W^{(1)}_{sj}=\left\{ \begin{aligned}
  e^{- t^{-1}\|\mathbf{u}_{s}-\mathbf{u}_{j}\|^{2}_{\mathbb{E}^{l}} },\;s\sim j \\
 0,\;s\nsim j   \\
\end{aligned}
 \right.  ,\;W^{(2)}_{sj}=\left\{ \begin{aligned}
 e^{- t^{-1}\|\mathbf{v}_{s}-\mathbf{v}_{j}\|^{2}_{\mathbb{E}^{p}} },\;s\sim j \\
 0,\;s\nsim j   \\
\end{aligned}
 \right..
\end{equation*}
  The  following matrix  is called   the matrix of the coupled mapping
\begin{equation}\label{10ys}
W=\alpha H+i \beta H,\;\alpha,\beta>0,\;\alpha+\beta=1,
$$
where
$$
H=\{H_{sj}\}\in \mathbb{C}^{n\times n},\;H_{sj} = \eta W^{(1)}_{sj}+\mu W^{(2)}_{sj} ,\;\eta,\mu>0,\;\eta+\mu=1.
\end{equation}
 Since the Hermitian components of the operator $W$ commute, we conclude that   the operator $W$ is normal.
 Note that the values of the parameters $\alpha,\beta,\eta,\mu$  reflect the influence of the corresponding components on the image. Let us construct the corresponding generalized  graph Laplacian operator $L=D-W.$     Let us  notice that in accordance with \eqref{10ys}, we have
$$
\theta:=\arg(D\mathbf{z},\mathbf{z})_{\mathbb{C}^{n}}= \arctan \left\{\frac{\mathrm{Im} (D\mathbf{z},\mathbf{z})_{\mathbb{C}^{n}}}{\mathrm{Re}(D\mathbf{z},\mathbf{z})_{\mathbb{C}^{n}}}\right\}=\arctan \frac{\beta}{\alpha},\;\mathbf{z}\in \mathbb{C}^{n}.
$$
Analogously, we get
\begin{equation}\label{19ys}
\arg(L\mathbf{z},\mathbf{z})_{\mathbb{C}^{n}} =\arctan \frac{\beta}{\alpha},\;\mathbf{z}\in \mathbb{C}^{n}.
\end{equation}
Note that the Hermitian components of the operator $L$ are diagonally dominant matrices, therefore  they are positive semidefinite matrices, we have
$$
\mathrm{Re}(L\mathbf{e}_{j},\mathbf{e}_{j})_{\mathbb{C}^{n}}>0,\;\mathrm{Im}(L\mathbf{e}_{j},\mathbf{e}_{j})_{\mathbb{C}^{n}}>0,\;j=1,2,\ldots,\xi.
$$
  Without loss of generality, assume that $l\geq p.$ Then the sets \eqref{9ys} can be put into correspondence with a subset of $\mathbb{C}^{l}$ as follows  $
\mathbf{x}_{j}=\mathbf{u}_{j}+i \mathbf{v}_{j},\,j=1,2,\ldots,n,
$
where we assume that the element $\mathbf{v}_{j}$ has the zero coordinates with the indices more than the value  $p<l.$ Thus,  if we find the mapping from $\mathbb{C}^{l}$ into $\mathbb{C}^{m},\;1\leq m \leq p,$ i.e.
$
F :\mathbb{C}^{l} \rightarrow \mathbb{C}^{m},
$
then due to the correspondence between the real and imaginary parts of the complex numbers   we    obtain    naturally    coupled mappings
$$
F_{1}:\mathbb{R}^{l} \rightarrow \mathbb{R}^{m},\;\;F_{2}:\mathbb{R}^{p} \rightarrow \mathbb{R}^{m}.
$$
  Consider the following problem
\begin{equation}\label{11ys}
\underset{ X^{\ast}D X=  e^{i\theta}I_{m}}{ \mathrm{argmin}}   \left|\sum\limits_{s,j=1}^{n}\|\mathbf{x}^{\cdot}_{s}-\mathbf{x}^{\cdot}_{j}\|_{\mathbb{C}^{m}}^{2}W_{sj}\right|,\,X^{\ast}D\mathbf{1}=0.
\end{equation}
Eventually, we can  formulate   the following theorem.
\begin{teo}\label{T4}  The solution of  problem \eqref{11ys} is represented by a matrix $Y\in \mathbb{C}^{n\times m},$ where
$$
 L\mathbf{y} _{j}=\lambda_{j}D\mathbf{y} _{j},\,j=1,2,\ldots,m.
$$
Moreover, the following coupled mapping is defined
$$
\mathbf{u}_{j}\rightarrow  \mathrm{Re}\,\mathbf{y}^{\cdot}_{j} ,\;\mathbf{v}_{j}\rightarrow \mathrm{Im}\,\mathbf{y}^{\cdot}_{j},\;j=1,2,\ldots,n.
$$
\end{teo}
\begin{proof}  In accordance with the definition,  $W^{(1)}, W^{(2)} \in \mathbb{R}^{n \times n}$ are the symmetric adjacency matrices corresponding to the considered datasets.
Consider a matrix
$
H = \eta W^{(1)} + \mu W^{(2)},
$
we have
$
W = (\alpha + i\beta) H,\;
D = (\alpha + i\beta) D_H,
$
where
$$
D_H = \operatorname{diag}\{d_1,\dots,d_n\}, \;d_s = \sum_{j=1}^{n} H_{sj}.
$$
Thus, the generalized graph Laplacian operator can be represented in the form
$$
L = D - W = (\alpha+i\beta)(D_H - H).
$$
Consider a related  generalized eigenvalue problem
\begin{equation}\label{12yg}
(D_H - H)\mathbf{e} = \lambda D_H \mathbf{e}.
\end{equation}
It is clear that the operator $D_H $ is positive definite, the operator $D^{-1}_H(D_H - H)$ is normal.
Applying the criterion established in Theorem \ref{T2}, we conclude that the latter problem possesses  a complete $D_{H}$  orthonormal  set $\{\mathbf{e}_{j}\}_{1}^{n}$   of the eigenvectors satisfying relation  \eqref{12yg}. The   corresponding   eigenvalues are nonnegative  including a single zero eigenvalue. Multiplying both sides of \eqref{12yg} on $ \alpha+i \beta ,$ we get
$$
 L\mathbf{e} _{j}=\lambda_{j}D\mathbf{e} _{j},\,j=1,2,\ldots,n.
$$
The orthogonality property gives us
$$
(D  \mathbf{e}_s, \mathbf{e}_j)_{\mathbb{C}^n} = \delta_{sj} e^{i \theta},\;s,j=1,2,\ldots,n.
$$
where we have chosen a normalized multiplier $|\alpha+i\beta|^{-1/2}$ so that $\mathbf{e}_j\rightarrow |\alpha+i\beta|^{-1/2}\mathbf{e}_j,\,j=1,2,\ldots,n.$ Consider the minimization problem \eqref{11ys}, rewritten in accordance with Lemma \ref{L1} in the following terms
\begin{equation}
\underset{ X^{\ast}D X=  e^{i\theta}I_{m}}{ \mathrm{argmin}}
\left| \operatorname{tr}(X^* L X) \right|,\,X^{\ast}D\mathbf{1}=0.
\end{equation}
It can be verified by direct calculations that the conditions $X^{\ast}D\mathbf{1}=0$  and $X^{\ast}D_{H}\mathbf{1}=0$ are equivalent. Applying Theorem \ref{T3}, Lemma \ref{L2}, we get
$$
\min \left|\sum\limits_{j=1}^{m} (L_{H} \mathbf{x} _{j},\mathbf{x} _{j})_{\mathbb{C}^{n}}\right|=
\sum\limits_{j=1}^{m}   \lambda_{j}  ,\;X^{\ast}D_{H}X= I_{m},\,X^{\ast}D_{H}\mathbf{1}=0.
$$
The latter relation can be rewritten consistently in the form
$$
\underset{X^{\ast}D X= (\alpha+i \beta)I_{m}}{\min} \left|\sum\limits_{j=1}^{m} (L  \mathbf{x} _{j},\mathbf{x} _{j})_{\mathbb{C}^{n}}\right|=
| \alpha+i \beta|\sum\limits_{j=1}^{m}   \lambda_{j};\;
\underset{X^{\ast}D X=  e^{i\theta}I_{m}}{\min}\left|\sum\limits_{j=1}^{m} (L  \mathbf{x} _{j},\mathbf{x} _{j})_{\mathbb{C}^{n}}\right|=
 \sum\limits_{j=1}^{m}   \lambda_{j}.
$$
Therefore
$$
\underset{ X^{\ast}D X=  e^{i\theta}I_{m}}{ \mathrm{argmin}}   \left|\sum\limits_{s,j=1}^{n}\|\mathbf{x}^{\cdot}_{s}-\mathbf{x}^{\cdot}_{j}\|_{\mathbb{C}^{m}}^{2}W_{sj}\right|= (\mathbf{e}_{1},\mathbf{e}_{2},\ldots,\mathbf{e}_{m})_{\mathbb{C}^{n}},\,X^{\ast}D\mathbf{1}=0.
$$
Thus, we have established the fact that the solution of the problem \eqref{11ys} exists and is represented in the stated form $Y=(\mathbf{e}_{1},\mathbf{e}_{2},\ldots,\mathbf{e}_{m}) .$ Having considered the following  correspondence, we obtain the desired result
 $$
 (\mathbf{u}_{1},\mathbf{u}_{2},...,\mathbf{u}_{n})\rightarrow \mathrm{Re}(\mathbf{y}^{\cdot}_{1},\mathbf{y}^{\cdot}_{2},\ldots,\mathbf{y}^{\cdot}_{n});\;
 (\mathbf{v}_{1},\mathbf{v}_{2},...,\mathbf{v}_{n})\rightarrow \mathrm{Im}(\mathbf{y}^{\cdot}_{1},\mathbf{y}^{\cdot}_{2},\ldots,\mathbf{y}^{\cdot}_{n}).
$$
The proof is complete.
\end{proof}

\subsubsection{Characteristic of the mapping}

Consider an abstract  approach to a function creating a weight matrix, assume that
$$
W_{sj}=\psi(\mathbf{x}_{s},\mathbf{x}_{j}),\;s,\!j=1,2,\ldots, n\,,
$$
where $\psi$ is a complex valued function. Thus, having imposed on the function $\psi$ corresponding conditions formulated in the previous paragraphs, we can create a mapping satisfying the required properties. It is remarkable that by virtue of the function $\psi$ we can give the sense, dictated by the relations between the elements $\mathbf{x}_{s}$ and $\mathbf{x}_{j},$ to the values $W_{sj},$ i.e.  the value $W_{sj}$ reflects the degree of a relationship between  the elements  $\mathbf{x}_{s}$ and $\mathbf{x}_{j}.$ Apparently, we can state that  Theorem \ref{T1} provides an indicator of the function $\psi$ in some sense. Assume that $X\in \mathbb{C}^{n\times n},\;X^{\ast}DX= e^{i\theta}I,$  applying Lemma \ref{L1}  observe the following relation
$$
  \mathrm{tr}\{ X ^{ \ast } L  X \}=\frac{1}{2} \sum\limits_{s,j=1}^{n}\|\mathbf{x}^{\cdot}_{s}-\mathbf{x}^{\cdot}_{j}\|_{\mathbb{C}^{n}}^{2}W_{sj}=\sum\limits_{j=1}^{n}       ( L\mathbf{x}   _{ j},       \mathbf{x}   _{ j })_{\mathbb{C}^{n}}.
$$
On the other hand, in accordance with Theorem \ref{T1},  we have
$$
 \mathrm{tr}\{ X ^{ \ast } L  X \}= e^{i\theta}\mathrm{tr}\left\{D^{-1}L\right\}= e^{i\theta}\sum\limits_{j=1}^{n}\left(1-\frac{W_{jj}}{D_{jj}}\right).
$$
It is remarkable that    $|W_{jj}/ D_{jj}|\leq 1.$ In accordance with the above, we have
$$
 \sum\limits_{s,j=1}^{n}\|\mathbf{x}^{\cdot}_{s}-\mathbf{x}^{\cdot}_{j}\|_{\mathbb{C}^{n}}^{2}W_{sj}=2 e^{i\theta}\sum\limits_{j=1}^{n}\left(1-\frac{W_{jj}}{D_{jj}}\right).
$$
Motivated by this phenomenon, we involve the functional $f(\psi)$ as the indicator of   the function $\psi$ influence on   the constructed mapping, where
$$
f(\psi):= e^{i\theta}\sum\limits_{j=1}^{n}\left(1-\frac{W_{jj}}{D_{jj}}\right)  .
$$
 The simple  heuristic reasonings  lead  us to the conclusion that    the smaller values $W_{sj},\,s\neq j,$ indicating a relationship between the elements, the more efficient the corresponding mapping from the point of view  of the elements closeness. It is rather clear that Theorem \ref{T1} can be used to reveal fully the properties of a function that   are relevant within the context. In order to observe  the particular case given by the weight matrix corresponding to distributions, assume that $\psi(\mathbf{x}_{s},\mathbf{x}_{j})=p_{sj},\;s,\!j=1,2,..., n,$ then in accordance with the   assumptions accepted above $D_{ss}=1,\,W_{ss}=0,$
therefore
$
f(\psi)=n e^{i\theta}.
$
The latter relation reflects  some intrinsic sense appealing to the connection of the  weight matrix and the dimension of the complex Euclidean space.   This rather simple  observation    fully reveals  the dependence of the mapping on the function $\psi$ since the last equality can be used as the standard.

\subsection{Unsupervised manifold alignment via the reproducing kernel}

Nowadays, many methods of sequencing individual cells are available, but difficulties arise when applying several different sequencing methods to the same cell. In the paper \cite{Jie Liu 2019}, the authors represent unsupervised manifold alignment algorithm  MMD-MA for integrating multiple measurements  of observed data performed on various aliquots of a given cell population.

  The idea of the   method is to map the cells measurement data     produced by various methods  into a  latent space.

 According to the   algorithm, the data of a single cell corresponding to several types of measurements are aligned by optimizing
an objective function consisting of three components:  1) maximum average discrepancy (MMD),
which sets constraints on the mapping  into the latent space, the "distances" between images should be minimal in the sense of some function determined by MMD,
 2) distortion term, a construction that preserves the structure of the initial data under the mapping, and  3) a penalty function that eliminates a trivial solution. It should be noted that the MMD-MA method does not require
any information on  the correspondence between the cell data processing methods.     The requirements for the  considered type of mapping are rather weak.
  They allow the algorithm to integrate data from measurements made in a single cell of heterogeneous features, such as gene expression, DNA availability, chromatin structure, methylation, and visualization data. In the paper  \cite{Jie Liu 2019} the authors
demonstrate the relevance of the MMD-MA method in simulated experiments and using a set of real data, including
data on gene expression in a single cell and data on methylation.

In this paragraph, we study the construction of the  mapping based on the concept of the reproducing kernel. In order to implement the classical approach,  the additional information from the theory of Hilbert spaces is provided. The issues related to the concept of the reproducing kernel  in the Hilbert space are considered and a theoretical justification of the opportunity to construct a suitable   mapping is given. Further, we study the objective function consisting of three terms considered  in the paper \cite{Jie Liu 2019}. In the conclusion, some ideas on    possibility of developing the mathematical concept of the method are represented.

\subsubsection{Reproducing kernel Hilbert space}

Let us define a complex linear space of elements of an arbitrary nature as a set of elements with a given structure of the linear operation acting into the initial set, i.e. the linear combinations of  the elements   belong to the initial set.
A linear space with a given inner (scalar) product operation is called the  unitary space.  Thus, the inner product naturally  defines the metric on the   elements of a linear space. However, it may happen that some  limits of sequences do not belong to the initial set of the linear space. Recall that a unitary space containing all its limiting elements is called  the Hilbert space. There are some discrepancies in the definitions of the Hilbert space. In some sources there is an additional condition related to the infinite dimension, i.e. an infinite-dimensional unitary space  is called   the Hilbert space.
As it can be seen from the definition given above, the Euclidean space is a finite-dimensional unitary space over the field of real numbers. Further, we will denote the abstract Hilbert space by $\mathfrak{H}$ and the corresponding inner product operation  by $(f,g)_{\mathfrak{H}},\,f,g\in\mathfrak{H}. $

Let $\mathfrak{A}$ be a set of elements of an arbitrary nature  and let $\mathfrak{H}$ be a system of complex functionals defined on the set  $\mathfrak{A}$ and forming a Hilbert space $\mathfrak{H}.$  A complex functional of two variables $K(x,y),\,x,y\in\mathfrak{A}$ is called   the reproducing kernel of the space $\mathfrak{H}$ if the following condition is satisfied. For an arbitrary fixed value  $x$, the functional $K(x,y)$ is an element of the space $\mathfrak{H},$ in symbols we have $ K_{x}(y):=K(x,y)\in \mathfrak{H},$
moreover, for an arbitrary  functional $f\in\mathfrak{H}$, we have
$$
f(x)=(f,K_{x})_{\mathfrak{H}}.
$$
 It is remarkable that there are such  representatives of the  Hilbert space of functionals defined on a set   that   a corresponding  reproducing kernel  does not exist.   The Hilbert space containing the reproducing kernel  is called   the reproducing kernel Hilbert space  (RKHS).    The most known representatives  of RKHS are the Bergman space and the Dirichlet space, i.e. the spaces of  holomorphic functions  on the unit disc.

 The following theorem establishes a criterion of the reproducing kernel existence.

\begin{teo}  A Hilbert space $\mathfrak{H}$ is RKHS if and only if   for an arbitrary element   $y\in\mathfrak{A}$ there exists a constant $C_{y}$ such that
$$
|f(y)|\leq C_{y}\|f\|_{\mathfrak{H}},\,f\in \mathfrak{H}.
$$
\end{teo}

We need the following  theorem.

\begin{teo}\label{T22} (Moore-Aronszajn)  Suppose that $K(x,y)$ is a  Hermitian symmetric, positive definite kernel on the set $\mathfrak{A}\times\mathfrak{A}.$ Then there exists a unique Hilbert space $\mathfrak{H}$ of functions defined on $\mathfrak{A},$ for which $K$ is the reproducing kernel. The space is defined as follows, consider the linear span of the functionals $K_{x},\,x\in \mathfrak{A},$ completing this set according to the norm generated by the following inner product, we obtain the desired RKHS
$$
\left(\sum\limits_{j=1}^{n}\alpha_{j}K_{x_{j}}, \sum\limits_{j=1}^{m}\beta_{j}K_{y_{j}}\right)_{\mathfrak{H}}:=\sum\limits_{j=1}^{n}\sum\limits_{k=1}^{m}\alpha_{j}\overline{\beta_{k}}K(x_{j},y_{k}),\;\alpha,\beta\in \mathbb{C}.
$$
\end{teo}
Note that the axioms of the inner  product can be verified directly taking into account the positive definiteness of the kernel $K.$
According to this definition, we have
$$
(K_{x},K_{y})_{\mathfrak{H}}=K(x,y),\,x,y\in\mathfrak{A}, \; f=\sum\limits_{j=1}^{\infty}\alpha_{j}K_{x_{j}},\,x_{j}\in \mathfrak{A},\,f\in \mathfrak{H},
$$
where   convergence is understood in the sense of  the norm of the  Hilbert space   $\mathfrak{H}.$

\subsubsection{Kernel based    mapping}
Consider the sets of elements $\mathfrak{X}_{1},\mathfrak{X}_{2},$ in general, these can be sets of elements of an arbitrary nature.   The elements of sets correspond to objects (cells, samples)   having a set of features   (various measurements reflecting one or more properties of the object). Thus, a correspondence naturally arises between a set of objects and a set of vectors, the coordinates of which are quantitative characteristics of features. We will consider the
subsets
 $$
X_{s}=\left(x^{(s)}_{1},x^{(s)}_{2},...,x^{(s)}_{n_{s}}\right)\subset\mathfrak{X}_{s},\;s=1,2.
$$
The numbers $n_{1},n_{2}$ reflect the number of objects under consideration. Note that we do not assume the existence of any similarity  between these subsets. At the same time, the idea of further reasonings is to construct a mapping  of these subsets into a certain space in which their images would be comparable.

As the main tool for further study, we use the following  positively defined Hermitian symmetric kernels
$\gamma_{s}(x,y),\,x,y\in \mathfrak{X}_{s}$  having  the range of values belonging to   $ \mathbb{C} .$   The biological meaning of these mathematical constructions may be to quantify the similarity between objects. Let us  introduce the following notation
$$
K^{(s)} =    \left\{K^{(s)}_{qj}\right\}:=\gamma_{s}\left(x^{(s)}_{q},x^{(s)}_{j}\right),\;q,j=1,2,...,n_{s},\,s=1,2,
$$
i.e. $K^{(s)}\in \mathbb{C}^{n_{s}\times n_{s}}.$
Since, in accordance with    the made assumptions, both kernels are Hermitian  symmetric, positive definite, then  using   Theorem \ref{T2}, we can  construct the  Hilbert spaces $\mathfrak{H}_{s}$ (RKHS) of functionals defined on $\mathfrak{X}_{s}$  for which $\gamma_{s}$ are reproducing kernels.

Following the idea of finding mappings  into some space in which the images of $\mathfrak{X}_{s}$ would be comparable, consider the following functionals
$$
z^{(s)}_{l}(x) =\sum\limits_{j=1}^{n_{s}}\alpha^{(s)}_{lj}\gamma_{s}\left(x^{(s)}_{j},x\right),\,x\in \mathfrak{X}_{s},\;l=1,2,...,p,\;p\in \mathbb{N}.
$$
 Note that in accordance with   Theorem \ref{T2}  these functionals are elements of the space
$\mathfrak{H}_{s}.$ Thus, the following  operator is defined
\begin{equation}\label{12y}
 F^{(s)}:\mathfrak{X}_{s}\rightarrow \mathbb{C}^{p},\;F^{(s)}x = \mathbf{z}^{(s)}(x),\,x\in \mathfrak{X}_{s},
\end{equation}
where
$$
   \mathbf{z}^{(s)}(x)=\left(z^{(s)}_{1}(x),z^{(s)}_{2}(x),...,z^{(s)}_{p}(x) \right)^{T}.
$$
It is clear that it is possible to implement the correspondence in this way, regardless of the nature of the set $\mathfrak{X}_{s},$ this universality provides the idea of the manifold alignment. Below, for the convenient form of writing we omit the  index $s,$  however the further  reasonings are correct in both $(s=1,2)$ cases.
Denote
$$
A \in \mathbb{C}^{ p \times n  },\;     A :=\left\{\alpha _{qj}\right\}.
$$
 Note that
\begin{equation}\label{13y}
A K     =\left\{z _{q}\left(x _{j}\right)\right\}\in \mathbb{C}^{ p \times n }.
\end{equation}
Using relation \eqref{13y}, applying Theorem \ref{T22}, we get
$$
A K A^{ \ast}=\left\{v _{qj}\right\},\;\,v _{qj} =\sum\limits_{k=1}^{n }z _{q}\left(x _{k}\right) \overline{\alpha} _{jk}
=\sum\limits_{k=1}^{n }\overline{\alpha}_{jk}    \sum\limits_{m=1}^{n }\alpha _{qm}\gamma\left(x _{m},x _{k}\right)=
$$
\begin{equation}\label{14y}
 =
 \left(  \sum\limits_{m=1}^{n }\alpha _{ jm}K_{x _{m}}, \sum\limits_{k=1}^{n }\alpha _{ qk}K_{x _{k}} \right)_{ \mathfrak{H} }  =
 \left( z _{j},z _{q}  \right)_{ \mathfrak{H} }.
\end{equation}
  Applying these  reasonings in  the both cases corresponding to the value of the  index $s,$ we can rewrite the latter relations in the form
\begin{equation}\label{15y}
A^{(s)} K^{(s)} =\left\{z^{(s)} _{q}\left(x^{(s)} _{j}\right)\right\}\in\mathbb{C}^{ p \times n_{s} },\;A^{(s)} K^{(s)} A^{(s) \ast}=\left\{  \left( z^{(s)} _{j},z^{(s)} _{q}  \right)_{ \!\! \! \mathfrak{H} }   \right\}\in \mathbb{C}^{p\times p}.
\end{equation}
The next challenge  is to select the mapping  parameters in order to minimize the  distances,   understood in the sense of a function of the Euclidean metric, between the images. Using the Gaussian  radial basic function (RBF), we will regulate matrices  $A^{(s)}$ under a certain condition connecting them and  determined by the relation  between the elements belonging to  $X_{s}.$     Thus, we can determine the similarity between the elements of the sets
$\mathfrak{X}_{s},\;s=1,2$ in the sense of the  distance  between them given by the following formula
$$
G^{(s,k)}_{qj} :=G\left\{\mathbf{z}^{(s)}\left(x^{(s)}_{q}\right),\mathbf{z}^{(k)} (x^{(k)}_{j} )\right\},\;\,s,k=1,2,\;q=1,2,...,n_{s},\;    j=1,2,...,n_{k},
$$
where the Gaussian   RBF  is defined as follows
$$
G(\mathbf{u},\mathbf{v}):=e^{- t^{-1}\|\mathbf{u}-\mathbf{v}\|_{ \mathbb{C}^{p}}^{2}  },\,\mathbf{u},\mathbf{v}\in  \mathbb{C}^{p},\,t\in\mathbb{R}\setminus\{0\}.
$$
The choice of the parameter $t $ is determined by a concrete  application. Note that the term distance is used in the heuristic sense since the construction does not satisfy the axioms of metric space, the verification is left to the reader.
It is clear that using the Gaussian   RBF, we can establish similarity in the sense of the  distance  between the elements $X_{s},$ for identical and different values of $s.$ This approach leads to the following function reflecting both the relationship between  elements within the sets and the relationship between  elements of different sets
$$
 \mathfrak{G}\left\{ A^{(1)}, A^{(2)}\right\}=\frac{1}{n^{2}_{1}}\sum\limits_{i,j=1}^{n_{1}}G^{(1,1)}_{ij}-\frac{2}{n_{1}n_{2}}\sum\limits_{i=1}^{n_{1}}\sum\limits_{j=1}^{n_{2}}G^{(1,2)}_{ij}+\frac{1}{n^{2}_{2}}\sum\limits_{i,j=1}^{n_{2}}G^{(2,2)}_{ij}.
$$
In accordance with  \cite{Jie Liu 2019}, the last formula has been  related to MMD construction \cite{Chwialkowski2015}. However, we use the exclusive notation since the similarity of definitions is rather vague.

Taking into account the specifics of the applications, the authors of the paper  \cite{Jie Liu 2019} consider that it is not sufficient  to find the minimum of the function $ \mathfrak{G}$ relative to the matrices $A^{(s)}.$ Thus, an additional term characterizing distortion
is introduced  in order to ensure the preservation of relations between the images of elements, i.e.
$$
\mathrm{dis}\left\{A^{(s)}\right\}=\left\|  K^{(s)}- K^{(s)}A^{(s)\ast} A^{(s) }K^{(s)} \right\|_{F},
$$
where the Frobenius  norm is used. This expression quantifies how much the matrix  $K^{(s)}$  differs   from the matrix
of inner products after mapping.  The restriction imposed  to $\mathrm{dis}\left\{A^{(s)}\right\}$ inherently guarantees that
the distortion between the data in the initial space and the  corresponding data after mapping should be small. Consider the following penalty function
$$
\mathrm{pen}\left\{A^{(s)} \right\}=\left\| I_{p}- A^{(s)}K^{(s)}A^{(s)\ast}  \right\|_{F}.
$$
Note that in accordance with   relation \eqref{14y}, we have
$$
I_{p}- A^{(s)}K^{(s)}A^{(s)\ast}  =\left\{\delta_{qj}-\left( z^{(s)}_{j}, z^{(s)}_{q}  \right)_{\!\!\mathfrak{H}_{s}}\right\},\;q,j=1,2,...,p.
$$
The latter relation characterizes the influence of the penalty function, thus it finds the mapping having coordinates  with the property close  to the orthonormal one.
 In order to ensure that the mapping to $\mathbb{C}^{p}$ is satisfied this property, we can add the penalty function as a summand.

Taking into account the above, we arrive at  the  problem of finding the minimum of the objective function
$$
\underset{ A^{(1)},A^{(2)}   }{ \mathrm{argmin}}\,\left( \mathfrak{G}\left\{ A^{(1)}, A^{(2)}\right\}  +\sum\limits_{s=1}^{2}\lambda_{1}\mathrm{dis}\left\{A^{(s)}\right\}+\lambda_{2} \mathrm{pen}\left\{A^{(s)}\right\}\right).
$$
The solution can be found due to   the Gradient Descent method \cite{Jie Liu 2019}. In practice, it is necessary to specify several parameters to solve the optimization problem. The parameters include    the dimension $p$ of the space $\mathbb{C}^{p},$   the parameter of Gaussian  RBF, as well as the parameters $\lambda_{1},\lambda_{2}.$ In the paper \cite{Jie Liu 2019} the authors assume that the parameters $p$ and $\lambda_{1},\lambda_{2}$ are set by the user.

Now assume that the set of elements  $\mathfrak{X}_{s}$ is endowed with the complex structure. In this case, we have a mapping naturally induced by   mapping \eqref{12y}, i.e.
$$
 F^{(s)}:\mathrm{Re}\mathfrak{X}_{s}\rightarrow \mathbb{R}^{p},\;F^{(s)}:\mathrm{Im}\mathfrak{X}_{s}\rightarrow \mathbb{R}^{p}.
$$
Apparently, we obtain a mapping implementing  the manifold alignment where a union of a four data sets  coupled in the natural way   are considered. The advantage of the kernel  based approach is in the following. Firstly, we should remark that the set  $\mathfrak{X}_{s}$ has  or does not have an arbitrary topological structure. It can be infinite-dimensional space endowed with the Hilbert space structure or finite set without any structure. This assumption makes a great  freedom in modifications of the method. In particular the interest arises in the case when $\mathfrak{X}_{s}$ is endowed with the quaternion structure     $\mathbb{H}.$ In this case the problem is to modify the reproducing kernel construction to construct the desirable mapping. Further generalizations in this direction  lead  us  to the hypercomplex numbers. The arbitrary structure of the preimages  is fully adopted to the implementation of the idea to unite several data sets of various nature. However, the concept of the mapping requires a concrete technique that should be invented in the general case corresponding to  the hypercomplex numbers.

\subsection{Prospective application to the modern methods}

The  mixOmics method \cite{Rohart2017} represents a specialized package for the $R$ language, developed for multivariate analysis of biological data with emphasis on data exploration, dimensionality reduction, and results visualization. The input data structure in mixOmics assumes use of matrix $X$ of size $N$ samples by $P$ predictors with continuous values and a categorical outcome vector $y,$ which is automatically converted to an indicator matrix $Y$ of size $N$ by $K$ classes. Output data represents a comprehensive set of results including latent components for projecting samples into lower-dimensional space, loading vectors demonstrating each feature's contribution, molecular signatures as selected features, classification and prediction results for new samples, as well as various visualization types including sample plots, variable plots, correlation networks, and heatmaps. All methods in the package are implemented as projection techniques where samples are summarized through $H$ latent components defined as linear combinations of original predictors. This approach ensures not only effective data dimensionality reduction but also preserves result interpretability, which is critically important for understanding biological mechanisms and identifying significant biomarkers in omics studies. However,  the general concept of mapping requires the usage   of the unified structure represented in this paper. The kernel based method can be involved as the   universal method related to the abstract data sets.

The  MultiVI method \cite{Ashuach2023} represents a deep generative model for probabilistic analysis of multimodal single-cell data, designed to integrate different molecular modalities into a unified representation. This method is developed for joint analysis of transcriptome, chromatin accessibility, and surface protein expression data from individual cells, even when information for some cells is available only for one or several modalities. MultiVI finds application in cellular diversity studies, characterization of cell types and states, as well as in data integration tasks from different laboratories and sequencing technologies. The method can combine single-cell RNA sequencing data (scRNA-seq), chromatin accessibility analysis (scATAC-seq), and surface protein measurements using antibodies (CITE-seq). MultiVI effectively works with both fully paired data, where all modalities are measured in the same cells, and partially paired or unpaired data, where different cells are characterized by different sets of modalities. The data processing workflow in MultiVI consists of several sequential stages. First, the method uses modality-specific encoders based on deep neural networks to create latent representations of each modality, including accounting for batch effects and technical data features. Then these representations are combined into a joint latent space through averaging with application of distance penalty between modalities. At the final stage, modality-specific decoders reconstruct observations from the latent representation using appropriate distributions for each data type. Input data is provided in count matrix format for each modality, where rows correspond to cells and columns to features (genes, genomic regions, or proteins). MultiVI generates low-dimensional joint cell representations, normalized and batch-corrected values for all modalities, as well as uncertainty estimates for imputed values. Here, we should remark that it is possible to involve the unified natural structure to create a low-dimensional representation. Using the coupled generalized Laplacian mapping, we can form an intermediate subspace of images and then apply the invented algorithm in order to implement the  coupling much  more accurate.

The MUON method \cite{Bredikhin2022} represents a multimodal analytical platform for omics data, developed for organizing, analyzing, visualizing, and sharing multimodal biological data. This tool is designed to solve computational challenges arising when working with multi-omics experiments, including efficient storage, indexing, and seamless access to large data volumes, tracking biological and technical metadata, as well as processing dependencies between different omics layers. MUON can combine various types of omics data including single-cell RNA sequencing (scRNA-seq), chromatin accessibility data (scATAC-seq), epitope profiling (CITE-seq), DNA methylation, as well as spatial omics data. The platform supports trimodal analyses such as scNMT-seq or TEA-seq and can process arbitrary numbers of data modalities. The processing workflow in MUON consists of several sequential stages. First, preprocessing of individual modalities occurs, including quality control, sample filtering, data normalization, and feature selection for analysis. Then dimensionality reduction methods are applied, which can work with individual modalities (e.g., principal component analysis) or jointly process multiple modalities through approaches such as multi-omics factor analysis (MOFA) or weighted nearest neighbors (WNN). At the next stage, cell neighborhood graphs are constructed based on obtained representations, which can use information from individual modalities or combined multimodal representations. The final stage includes creating nonlinear embeddings through UMAP-type methods and clustering for cell type identification. Input data is provided in count matrix format for each modality along with corresponding metadata. MUON uses the MuData container, which represents a hierarchical data structure where each omics modality is stored as an AnnData object.  Output data includes processed count matrices, low-dimensional representations, cell embeddings, neighborhood graphs, cluster labels, and differential analysis results, which can be visualized and used for biological interpretation of results. Apparently,  the concept of embedding used in the method  admits application of  the   methods elaborated in    the paper.

The PolarBear method \cite{Zhang2022}  represents a semi-supervised machine learning model for predicting missing modalities and aligning single-cell data between different types of omic measurements. PolarBear's main task lies in solving the problem of integrating multimodal single-cell data when most cells have only one type of measurement available (e.g., only transcriptome or only chromatin accessibility), while comprehensive co-assay data measuring multiple modalities in a single cell are available in limited quantities. The method can be applied in cellular regulation studies, cellular heterogeneity analysis, differential gene expression studies between cell types, identification of cell-specific regulatory elements in tumor samples, and generating hypotheses about biological processes in modalities inaccessible for direct measurement. PolarBear can combine scRNA-seq and scATAC-seq data, working with various types of co-assay technologies including CAR-seq, SNARE-seq, Paired-seq, and SHARE-seq. Principally, the method can be adapted to work with other types of omic data since it does not require feature correspondence between different measurement modalities. Data processing in PolarBear occurs in two main stages. At the first stage, the method trains two separate beta-variational autoencoders (beta-VAE) for each data modality, using both paired co-assay data and much more numerous unimodal data from public databases. The autoencoders learn to create stable latent cell representations independent of sequencing depth and batch effects. For scRNA-seq, the autoencoder assumes that gene read counts follow a zero-inflated negative binomial distribution, while for scATAC-seq the Bernoulli distribution is used for binary chromatin accessibility peaks. At the second stage, a fully connected translator layer is added between the trained autoencoders, which is trained in supervised mode exclusively on co-assay data to translate between latent representations of two modalities in both directions. Input data is provided as raw gene count matrices for scRNA-seq and binarized peak count matrices for scATAC-seq, with additional information about batches and sequencing depth for each cell. Output data includes predicted profiles in the missing modality, sequencing depth-normalized expression estimates, latent cell representations for alignment between modalities, differential expression analysis results, and cell-specific marker genes, providing possibilities for biological interpretation and subsequent analysis of integrated multimodal data. The alignment between modalities creates the prerequisite for the paper results application.
However, the nature of the semi-supervised model and the only type of measurements may create some obstacles to the direct application of the methods discussed in the paper.

The sciCAN method \cite{Xu Y 2022} represents a method for integrating single-cell chromatin accessibility and gene expression data using cycle-consistent adversarial networks. The method is designed to combine scATAC-seq and scRNA-seq data into a unified representation without requiring prior information about cell correspondence between modalities. The sciCAN method  can be applied for hematopoietic hierarchy analysis, studying cellular responses to CRISPR perturbations, constructing joint developmental cell trajectories, transferring cell type labels between modalities, and other integrative single-cell data analyses. The method can combine single-cell RNA sequencing and single-cell ATAC sequencing data, transforming chromatin accessibility peak matrices into gene activity matrices to ensure compatibility with gene expression data. Input data undergoes logarithmic transformation normalization with pseudocount addition, after which top-3000 highly variable genes are identified for each modality and used as features for integration. Data processing in sciCAN consists of two main components: representation learning and modality alignment. Encoder $E$ projects high-dimensional data from both modalities into a joint low-dimensional space using noise contrastive estimation loss function to learn discriminative representation. For modality alignment, two separate discriminator networks are used: Drna distinguishes modality source in latent space, while Datac works with generator $G$ to create connections between modalities through adversarial training with addition of cycle-consistent losses. The architecture includes fully connected layers with batch normalization and ReLU activation for the encoder, three-layer multilayer perceptrons for discriminators with sigmoid activation, and a two-layer decoder for the generator. Additionally, linear transformation of 128-dimensional latent representation to 32-dimensional output and 25-dimensional SoftMax-activated output for NCE loss computation is applied. Input data is provided in gene expression and gene activity matrix format after preprocessing and normalization, while output data represents 128-dimensional joint latent cell representation used for subsequent integrative analysis including trajectory construction, label transfer, and clustering of cells from both modalities in unified feature space. In this case, we should emphasize that the combination of scATAC-seq and scRNA-seq data can be considered in the framework of the paper results. In particular, heterogeneous data can be successfully coupled by the kernel based method.

The SCIM method (Single-Cell data Integration via Matching) represents a scalable approach for integrating single-cell data obtained using different profiling technologies without requiring feature correspondence between modalities. The method is designed to recover correspondences between cells measured by different technologies when paired correspondences between data sets are lost due to cell consumption during profiling processes. The SCIM method  can be applied to integrate any single-cell omics technologies including scRNA-seq, CyTOF, proteomics, genomics  provided there is common latent structure in the data. The method can combine data from two or more technologies that measure non-overlapping feature sets, for example expression of different gene sets, gene expression with image characteristics, or any other single-cell measurements originating from the same cell suspension. The  SCIM method consists of two main processing stages: first, an integrated latent space invariant to technology is created using an autoencoder/decoder framework with adversarial objective function, where separate encoder and decoder networks are trained for each technology, and a single discriminator acts in latent space to ensure indistinguishability of representations from different technologies. Then a bipartite matching scheme is applied for pairwise connection of cells between technologies, using their low-dimensional latent representations through an efficient bipartite matching algorithm based on minimum cost maximum flow problem. Input data is provided in cell-by-feature matrix format for each technology, where features are specific to the profiling technology but can represent gene expression, protein levels, or other measurements. Output data represents eight-dimensional latent cell representations in common feature space and paired correspondences between cells from different technologies obtained through bipartite matching algorithm, permitting  use of true observed signals per cell pair for any subsequent analysis while maintaining technology-invariant data integration. It should be noted that the specifics of the conditions, such as the use of various profiling technologies and the need for paired correspondences between data sets, determine the relevance of the paper results application  and subsequent comparison analysis.

  From the point of view of applications, the aim of this paper is to develop methods for the diagonal integration of multimodal biological data for a comprehensive study of the biological processes in particular the aging process. The prospective  project based on the paper fundamental results  plans to develop new mathematical methods of diagonal integration that take into account the identified limitations of existing approaches. We propose to combine and jointly analyze seven different types of data: genomic, transcriptomic, proteomic, metabolomic, and lipidomic data. The aging process acts as a suprasystem factor  for diagonal integration as a fundamental biological phenomenon that manifests itself at all levels of the molecular organization. A comprehensive analysis of the integrated data will make it possible to identify new pharmaceutical targets for therapeutic interventions in the aging process, identify multidimensional biomarkers of aging with increased predictive power, and establish causal relationships between molecular changes at various levels of biological organization in the context of age-associated processes.

\section{Conclusions}

  Having  analyzed  the main mathematical principles forming  the concept of the unsupervised topological alignment, we represent a natural algebraic structure coupling the heterogeneous  data sets as well as their images. In this regard, the harmonious generalization of the graph Laplacian  method was obtained over the field of the complex numbers. The  kernel based method with the central idea to find an appropriate  structure coupling images of data sets of various nature was extended to the complex vector space.
The prospective theoretical results appeal  to  more complicated algebraic structures allowing one  to   couple  naturally an arbitrary number of data sets and their images. Thus, the quaternion structure can be considered as  further generalization that leads to the  hypercomplex numbers and more abstract mathematical objects.
 It was shown that the kernel based methods are completely efficient for the implementation of the idea to unite  data sets of various nature. The main obstacle in further development is that a mapping requires a concrete technique,  i.e. generalizations of the well-known results of the operator theory for the module over the hypercomplex structure. Finally, we represented the detailed analysis of the modern  biological methods supplied with the comments on  possible applications of the invented approach.
The authors   believe  that the represented theoretical  approach is  principally novel while the obtained conclusions admit relevant  biological applications.

\end{document}